\documentclass{amsart}
\UseRawInputEncoding
\usepackage{amssymb}
\usepackage{graphicx}
\usepackage{cite}
\usepackage[backref=page]{hyperref}
\usepackage{xcolor}
\hypersetup{urlcolor=blue, citecolor=red}

\newtheorem{theorem}{Theorem}[section]

\newtheorem{proposition}[theorem]{Proposition}
\newtheorem{corollary}[theorem]{Corollary}
\theoremstyle{definition}

\theoremstyle{remark}
\newtheorem{remark}[theorem]{Remark}

\numberwithin{equation}{section}

\begin{document}

\title{Notes on Conformal Perturbation of Heat Kernels}

%    Information for first author
%\author{*}
%    Address of record for the research reported here
%\address{}
%    Current address
%\curraddr{}
%\email{}
%    \thanks will become a 1st page footnote.
%\thanks{The first author was supported in part by NSF Grant \#000000.}

%    Information for second author

\author{Shiliang Zhao}
\address{Department of Mathematics, Sichuan University, Chengdu, Sichuan 610064, P.R.China}
\email{zhaoshiliang@scu.edu.cn}
%\thanks{The author is supported by the National Natural Science Foundation of China under Grant No.11901407}

%    General info

%\date{January 1, 2001 and, in revised form, June 22, 2001.}

%\dedicatory{This paper is dedicated to our advisors.}


\begin{abstract}
Let $(M, g)$ be a smooth n-dimensional Riemannian manifold for $n\ge 3$. Consider the conformal perturbation $\tilde{g}=h g$ where $h$ is a smooth bounded positive function  on $M$. Denote by  $\widetilde{\Delta}$   the Laplace-Beltrami operator of manifold $(M, \tilde{g})$.  In this paper, we derive the upper bounds of the heat kernels for $(-\widetilde{\Delta})^\sigma$ with $0< \sigma \le 1$. Moreover, we also investigate the  gradient estimates of the heat kernel for $\widetilde{\Delta}$.
\end{abstract}

\maketitle

\section{Introduction}
Let $(M, g)$ be a smooth n-dimensional Riemannian manifold with metric $g$ for $n\ge 2$. Moreover set the volume element $d\mu$ and the Laplace-Beltrami operator $\Delta$. Denote by $d(x, y)$ the geodesic distance for $x, y \in M$ and $B(x, r)$ the geodesic ball centered at $x$ with radius $r>0$. In local coordinates the Laplace-Beltrami operator can be expressed as
$$\Delta= \frac{1}{\sqrt{G}} \partial_i (g^{ij} \sqrt{G}\partial_j), $$
where $G=\det (g_{ij})$, $(g^{ij})=g^{-1}$ and we have used the Einstein summation convention. Denote the heat kernel on $M$ by $p_t(x,y)$ which is the integral kernel of semigroup $e^{t\Delta}$  acting on $L^2(M, d\mu)$.

On complete manifolds with non-negative Ricci curvature, Li and Yau(\cite{L-Y86}) proved the upper and lower Gaussian bounds as well as the gradient estimates for the heat kernel as follows:
\begin{equation*}
\text{(LY)} \hspace{0.5cm}  \frac{C_1}{V(x, \sqrt{t})} \exp \left( -\frac{d^2(x, y)}{c_1 t} \right) \le p_t(x, y) \le  \frac{C_2}{V(x, \sqrt{t})} \exp \left( -\frac{d^2(x, y)}{c_2 t} \right),
\end{equation*}
and
\begin{equation*}
 \text{(G)} \hspace{1.5cm} |\nabla_x p_t (x, y)| \le \frac{C}{\sqrt{t} V(x, \sqrt{t})} \exp \left( -\frac{d^2(x, y)}{c t} \right),
\end{equation*}
for all $t>0, x, y \in M$. Since then, much effort has been made to establish (LY) and (G) on more general settings.  By \cite{FS86,HS01}, (LY) holds if and only if the manifolds satisfy the parabolic Harnack inequality. Moreover, it has been proved   by \cite{G-SC} that the parabolic Harnack inequality is stable under rough isometries and hence (LY) holds under rough isometries. However, when considering  such stability results for (G), the situation is more subtle. Recently, in \cite{D} Devyver studied the gradient estimates under  assumption on the negative part of the Ricci curvature. In \cite{CJKS}, Coulhon et al proved several  equivalent characterisations of (G). Among others, the $L^\infty$-reverse H\"older inequality for the gradients of harmonic functions is equivalent to (G) under some assumption on the underlying spaces. For more results, we refer the readers to \cite{BM18,Davies90,G91,G,HS01,J-L,Li06,Li12,SC92a,SC92} and references therein.

The  aim of this paper is to study  the upper bounds and gradient estimates of the  heat kernel under conformal perturbation. This problem has applications in both mathematics and physics. We refer the readers to \cite{A,B-R,D,G,Li06,Li18,J-L,M} and references therein.

Now consider the conformal perturbation of the metric. Let $\tilde{g}= h(x)g$ for $\forall x\in M$ where $h(x)$ is a positive smooth function on $M$. Denote by $d\widetilde{\mu}$ , $\widetilde{\Delta}$, $\widetilde{p}_t(x,y)$ the corresponding volume element, Laplace-Beltrami, heat kernel of $e^{t\widetilde{\Delta}}$ defined on $L^2(M, d\widetilde{\mu})$.  Note that  in local coordinates, we have
$$ \widetilde{\Delta}=  \frac{1}{h^{\frac{n}{2}}\sqrt{G}} \partial_i (h^{\frac{n}{2}-1} g^{ij} \sqrt{G}\partial_j). $$

Thus when $n=2$, it follows $\widetilde{\Delta}= h^{-1}\Delta$.  In \cite{M}, Morpurgo used the Dyson series to study heat kernel $\widetilde{p}_t(x,y)$ where the Dyson series are determined by $p_t(x,y)$ as well as $h(x)$  and hence one can get the upper bounds of $\widetilde{p}_t(x,y)$ from that of $p_t(x,y)$.  More generally, the methods works for the semigroup $e^{th^{-1}\Delta}$ defined on $L^2(M, h d\mu)$ of compact n-dimensional Riemannian manifold where $h$ is a positive smooth function. However, as said by the author, the conformal-geometric significance of the results is only in dimension 2.

As is shown,  $\widetilde{\Delta}= h^{-1}\Delta$ does not hold for $n\ge 3$. Indeed, the Laplace-Beltrami operator $\widetilde{\Delta}$ is related to the  weighted Laplace operator on $(M, g)$. Precisely,
$$ \widetilde{\Delta}= h^{-1}  \frac{1}{h^{\frac{n}{2}-1}\sqrt{G}} \partial_i (h^{\frac{n}{2}-1} g^{ij} \sqrt{G}\partial_j) \triangleq h^{-1} \bar{\Delta}, $$
where $\bar{\Delta}$ is the weighted Laplace operator on the weighted manifold $(M, g, d\bar{\mu})$ with $d\bar{\mu}= h^{\frac{n}{2}-1} d\mu$.  Denote the weighted heat kernel  by $\bar{p}_t(x,y)$ which is the integral kernel of $e^{t\bar{\Delta}}$ acting on $L^2(M, d\bar{\mu})$ where $d\bar{\mu}=h^{\frac{n}{2}-1} d\mu$. Thus according to the argument in \cite{M}, the Dyson series are determined by $\bar{p}_t(x,y)$ on the weighted manifold $(M, g, d\bar{\mu})$ instead of $p_t(x,y)$ on $(M, g, d\mu)$.

To proceed, we recall some facts about the weighted manifolds. Note that weighted manifolds have been extensively studied in recent years and found various  applications in many areas such as geometric analysis, Markov diffusion theory. See for example \cite{BGL, G, G-SC, Li05, Li18, WW09}.

Now we recall some facts about the weighted manifolds. First we call the manifold satisfies the doubling condition, if there exits constant $C>0$ such that
$$  V(x, 2r)\le C V(x, r), \hspace{0.5cm} \forall x\in M, r>0, $$
where $V(x, r)$ is the volume of ball $B(x, r)$ with respect to measure $d\mu$.

Furthermore, we need some notations from \cite{G, G-SC}. For a manifold $(M, g)$, fix a point $o\in M$ and set
$$  |x| \triangleq d(x, o), \hspace{0.5cm} V(s) \triangleq V(o, s). $$
The manifold is said to have relatively connected annuli (RCA) if there exists a constant $K>0$ such that for all $x, y \in M$ and large enough $r$ with $|x|=|y|=r$, there exists a continuous path $\gamma: [0, 1] \rightarrow M$ with $\gamma(0)=x,   \gamma(1)=y$ whose image is contained in $B(o, Kr)\backslash B(o, K^{-1}r)$.
Then the following results hold.
\begin{theorem}
Let $M$ be a complete n-dimensional Riemannian manifold for $n\ge 3$ with non-negative Ricci curvature satisfying (RCA) with respect to $o\in M$. Let $h$ be a positive bounded smooth function on $M$.
Then  there exist constants $\alpha, C_1>0, C_2>1$ such that for  $\|\phi \|_\infty < C^{-1}_2 $ and $\forall t>0, x, y\in M$,
$$  \widetilde{p}_t(x,y) \le \frac{C_1}{1-C_2 \|\phi \|_\infty} \frac{1}{V(x, \sqrt{\alpha t}) } \exp \left(-\frac{d^2(x, y)}{\alpha t}\right),  $$
where $\phi(x)=1-h(x)$.
\end{theorem}
Now we give some remarks about the assumptions on $\phi$.
\begin{remark}
1. Compared to the previous results \cite{G, J-L, M, SC92}, our contributions are twofold. First we generalize the results in \cite{M} to noncompact manifolds and hence provide an alternative approach to study the heat kernel under conformal perturbation. Second, the upper bounds here are expressed in terms of geodesic and volumes with respect to the origin metric $g$.

2. Note  that the results can be equivalently stated as
$$  \widetilde{p}_t(x,y) \le \frac{C_1}{1-C_2 \|\phi \|_\infty} \bar{p}_{\alpha t} (x, y). $$
The assumption $\|\phi \|_\infty <C_2^{-1}$ implies that $h$ is bounded. Next example shows that when $h$ is unbounded, the results in Theorem 1.1 may not hold. Let $M$ be the two dimension Euclidean spaces $\mathbb{R}^2$  with $g_{ij}= \delta_{ij}$ and $h(x)=(1+|x|^2)^{\frac{1}{2}}$. According to  \cite[Theorem 1.2]{P}, the on diagonal estimates of $\widetilde{p}$ satisfies
$$  c_1 t^{-1} (1+|x|)^{4}\le \widetilde{p}_t (x, x), \hspace{0.5cm} \forall~ 0<t \le c_2 (1+|x|^2)^{\frac{1}{2}}. $$
Meanwhile, when $n=2$, $\bar{p}_t(x, y)= c t^{-1} e^{-\frac{|x-y|^2}{4t}}$ is the classical heat kernel of $\mathbb{R}^2$.
\end{remark}
When $ V(x, r) \simeq r^n $ for $\forall x\in M, r>0$, by the Bochner¡¯s subordination principle, we obtain the following result.
\begin{corollary}
Let $0<\sigma<1$. Denote by $\widetilde{p}_t^\sigma(x,y)$ the heat kernel for $(-\widetilde{\Delta})^\sigma$. Assume that $ V(x, r) \simeq r^n $ for $\forall x\in M, r>0$. Then under the assumption of Theorem 1.1, there exists $C>0$ such that
$$ \widetilde{p}_t^\sigma(x,y)\le C t^{-\frac{n}{2\delta}} \wedge \frac{t}{d(x, y)^{n+2\delta}}, \hspace{0.5cm} \forall x, y\in M, t>0. $$
\end{corollary}

To get the gradient estimates, we need some further assumptions. Let $H(x)=h^{\frac{n-2}{4}}(x), W(x)=\frac{\Delta H(x)}{H(x)}$. Set
$$ \text{(A)} \hspace{1cm} \left|\frac{\nabla H(x)}{H(x)}\right|, |W(x)|, |\nabla W(x)|, |\Delta W(x)| \le C, \hspace{0.5cm} \forall x\in M. $$
\begin{theorem}
Let $M$ be a complete n-dimensional Riemannian manifold for $n\ge 3$ with non-negative Ricci curvature satisfying (RCA) with respect to $o\in M$. Let $h$ be a positive bounded smooth function on $M$. Suppose that $H(x)$ satisfies (A). Then  there exist constants $\gamma, C_1>0, C_2>1$ such that for  $\|\phi \|_\infty < C^{-1}_2 $ and $\forall t>0,  x, y\in M$,
$$ | \widetilde{\nabla}_x \widetilde{p}_t(x,y)| \le \frac{C_1}{1-C_2 \|\phi \|_\infty}\frac{1}{1\wedge \sqrt{t}} \frac{1}{V(x, \sqrt{\gamma t}) } \exp \left(-\frac{d^2(x, y)}{\gamma t}\right),  $$
where $\phi(x)=1-h(x)$.
\end{theorem}

In this paper, we use the following notations.
For positive functions $f(x)$ and $g(x)$ defined on  $M$, we say $f\simeq g$ if there exists a constant $C>0$ such that
$ C^{-1}\le \frac{f(x)}{g(x)} \le C , \forall x\in M.  $
Set $f\wedge g (x) = \min \{ f(x), g(x) \}$ for $\forall x\in M$ where $f(x), g(x)$ are two functions defined on $M$.
The constants $c, C>0$ may change from line to line, unless otherwise stated.

\section{Preliminaries}
In this section,  we collect several estimates of the heat kernel $\bar{p}_t(x, y)$ on weighted manifolds which will be used in the sequel.

Note first that under the assumption of Theorem 1.1, 1.4, we always have $0<1-C_2^{-1}<h(x)<1+C_2^{-1}$ for $\forall x\in M$. Thus by \cite[p.861-863]{G-SC} we have $ \bar{\mu}(B(x, r)) \simeq V(x, r)$ and hence $ (M, d, d\bar{\mu}) $ satisfies the doubling property.
Moreover, by \cite{G,G-SC}, the following Li-Yau type estimates hold:
 \begin{equation}\label{weighted heat kernel real time}
    \frac{C_1}{\overline{V}(x, \sqrt{t})} \exp \left( -\frac{d^2(x, y)}{c_1 t} \right) \le \bar{p}_t(x, y) \le  \frac{C_2}{\overline{V}(x, \sqrt{t})} \exp \left( -\frac{d^2(x, y)}{c_2 t} \right)
 \end{equation}
where $\overline{V}(x, r)=\bar{\mu}(B(x, r))$.
Moreover, by doubling property, we have for $0<\alpha<\beta$
\begin{equation}\label{comparsion}
  \bar{p}_{\alpha t}(x, y)\le C \bar{p}_{\beta t}(x, y), \hspace{1cm} \forall t>0, x, y\in M,
\end{equation}
where $C$ is determined by $\alpha, \beta$.
Moreover, we have the following results.
\begin{proposition}
Let $0\le \theta_0 < \frac{\pi}{2} $. Under the assumption of Theorem 1.1, there exist $C, c>0$ such that
\begin{equation}\label{complex time estimate}
  |\bar{p}_z(x, y)|\le C \bar{p}_{c|z|} (x, y), \hspace{0.5cm} \forall ~ x, y \in M, |\arg z|\le \theta_0,
\end{equation}
where $C, c$ are determined by $\theta_0$.
\end{proposition}
\begin{proof}
According to \cite[Corollary 4.4]{C-S}, there exist constants $C, \delta>0$ such that
$$  |\bar{p}_z(x, y)|\le \frac{C \left( 1+ \Re \frac{d^2(x, y)}{z}\right)^\delta }{\sqrt{ \overline{V}(x,\sqrt{\frac{|z|}{\cos \theta}}) \overline{V}(y,\sqrt{\frac{|z|}{\cos \theta}})} } \exp \left(-c~ \Re \frac{d^2(x, y)}{z}\right) \frac{1}{(\cos \theta)^\delta}  $$
for all $x, y \in M, \Re z>0$ where $\theta=\arg z$. Then by the doubling property of $(M, d, d\bar{\mu})$ and the fact $\bar{\mu}(B(x, r)) \simeq V(x, r) $, we have
$$|\bar{p}_z (x, y)| \le \frac{C}{V\left(x,\sqrt{\frac{|z|}{\cos \theta}}\right)  } \exp \left(-c~ \Re \frac{d^2(x, y)}{z}\right) \frac{1}{(\cos \theta)^\delta}. $$
Note that, $\Re \frac{d^2}{z}= \frac{d^2 \cos \theta}{|z|}$ where $\theta= \arg z$. Finally by the doubling property of $(M, d, d\bar{\mu})$ and \eqref{comparsion}, the result holds.
\end{proof}
\begin{proposition}
Under the assumption of Theorem 1.4, there exist constants $C, c>0$ such that
\begin{equation*}%\label{weighted heat kernel complex time}
  |\nabla_x \bar{p}_t (x, y)| \le  \frac{C}{1\wedge \sqrt{t}} \bar{p}_{ct} (x, y) \hspace{1cm} \forall t>0, x, y\in M.
\end{equation*}
\end{proposition}

\begin{proof}
By Doob transform, we have
$$ \bar{p}_t(x, y) = \frac{1}{H(x) H(y)} p^{W}_t (x, y)  $$
where $p^{W}_t (x, y)$ is the integral kernel of $e^{t\Delta_W}$ defined on $L^2(M, d\mu)$ and $\Delta_W = \Delta - W $ with $W(x)=\frac{\Delta H(x)}{H(x)}$. Therefore under the assumption (A), we obtain by \cite[Theorem 3.1]{L-Y86},
\begin{align*}
 \left| \nabla_x \bar{p}_t (x, y) \right|  & = \left| -\frac{\nabla_x H(x) }{H(x)} \bar{p}_t(x, y) + \frac{\nabla_x  p^W_t(x, y) }{H(x)H(y)} \right|  \\
   & \le   \frac{C}{1 \wedge \sqrt{t}} \bar{p}_t (x, y) + H^{-1}(x)H^{-1}(y) \sqrt{|\partial_t p^W_t(x, y)| p^W_t(x, y)} \\
   & \le \frac{C'}{1 \wedge \sqrt{t}} \bar{p}_{ct} (x, y).
\end{align*}
In the last inequality, we have used the equality $\partial_t p^W_t(x, y)= H(x) H(y) \partial_t  \bar{p}_t(x, y) $. The time derivative estimates for $\bar{p}_t(x, y)$ can be found in \cite{G}.
\end{proof}

\section{Proof the main results}
Now we are ready to prove our main results.
\begin{proof}[Proof of Theorem 1.1]
First we give the Dyson series for $\widetilde{p}_t(x, y)$.  Now we claim that
\begin{equation}\label{Dyson series}
  \widetilde{p}_t(x, y) = \sum_{k=0}^{\infty} \partial_t^k \beta_t^k(x, y), \hspace{0.5cm} \forall ~ t>0, x, y \in M.
\end{equation}
Set $ \beta_z^0(x, y)= \bar{p}_z(x, y)   $  and denote  inductively
$$  \beta_z^k(x, y)= \int_0^1 \int_M z\bar{p}_{(1-v)z} (x, w) \beta_{vz}^{k-1} (w, y) \phi(w) d\bar{\mu}(w) dv, $$
where $k\ge 1$, $\Re z>0$ , $\phi(w)=1-h(w)$ for all $ w\in M $.

By changing variable $s=v|z|$, we have
\begin{align*}
  |\beta_z^1(x,y)| & = \left| \int_0^{|z|} \int_M  e^{i\theta} \bar{p}_{(|z|-s)e^{i \theta}} (x, w) \bar{p}_{se^{i \theta}} (w, y)   \phi (w)  d\bar{\mu} ds \right| \\
   & \le \int_0^{|z|} \int_M |\bar{p}_{(|z|-s)e^{i \theta}} (x, w) \bar{p}_{se^{i \theta}} (w, y) | \| \phi \|_\infty d\bar{\mu} ds \\
  &  \le C \int_0^{|z|} \int_M \bar{p}_{\alpha (|z|-s)} (x, w) \bar{p}_{\alpha s} (w, y)  \| \phi \|_\infty d\bar{\mu} ds   \\
   & \le C |z| \| \phi \|_\infty \bar{p}_{\alpha |z|} (x,y).
\end{align*}
We have used \eqref{complex time estimate} and the semigroup properties in the last inequality.

By induction for $k\ge 2$, we have for all $  x, y \in M, \Re z>0, |\arg z|\le \frac{\pi}{4}$
\begin{align*}
  |\beta_z^k (x,y)| & \le  \int_0^{|z|} \int_M |\bar{p}_{(|z|-s)e^{i \theta}} (x, w) \beta^{k-1}_{se^{i \theta}} (w, y) | \| \phi \|_\infty d\bar{\mu} ds   \\
       & \le C^{k-1} \| \phi \|_{\infty}^k \frac{1}{(k-1)!} \int_0^{|z|} s^{k-1} ds \int_M \bar{p}_{\alpha(|z|-s)}(x,w) \bar{p}_{\alpha s}(w,y) d\bar{\mu} \\
       & \le C^k \| \phi \|_{\infty}^k \frac{|z|^k}{k! } \bar{p}_{\alpha |z|} (x,y).
\end{align*}
By the Cauchy's integral formula and the above estimates we can give the estimates for $\partial_t^k \beta_t^k(x, y)$. To be precise,   for any $t>0$, consider the circle $\Gamma$  centered at $t$ with radius $t\sin \frac{\pi}{4}$. Thus  we obtain by Cauchy's integral formula
$$ |\partial_t^k\beta_t^k(x, y)|  \le \frac{k!}{(t\sin \frac{\pi}{4})^k} \max_{z\in \Gamma} |\beta_z^k(x, y)| \le 2^{\frac{k}{2}} C^k \|\phi  \|_{\infty}^k \max_{z\in \Gamma}  \left| \frac{z}{t} \right|^k \bar{p}_{\alpha|z|} (x, y).  $$
Since $|z-t|=t\sin \frac{\pi}{4}$, it follows
$$  |\partial_t^k\beta_t^k(x, y)| \le 2^{2+\frac{k}{2}} C^k \|\phi  \|_{\infty}^k \max_{z\in \Gamma}   \bar{p}_{\alpha|z|} (x, y) \le C_1 (\sqrt{2}C \|\phi \|_\infty)^k \bar{p}_{\alpha' t} (x, y), $$
where we have used \eqref{weighted heat kernel real time} in the last inequality. As a result, the right hand side of  \eqref{Dyson series} converges for all $x, y \in M$. Moreover, it indicates
\begin{equation}\label{series estimate}
  |\sum_{k=0}^\infty \partial_t^k \beta_t^k(x, y)| \le \frac{C_1}{1-\sqrt{2}C \|\phi \|_\infty}  \bar{p}_{\alpha' t} (x, y).
\end{equation}
We are only left to show the claim \eqref{Dyson series} holds. Indeed, by \eqref{series estimate} the right hand side of \eqref{Dyson series} converges.  Then we have
$$ \sum_{k=0}^\infty \widetilde{\Delta} \partial_t^k \beta_t^k =  \sum_{k=0}^\infty \partial_t^k [ \partial_t\beta_t^k -\phi \beta_t^{k-1}  ] = (1-\phi)\sum_{k=0}^\infty  \partial_t^{k+1} \beta_t^k. $$
Thus the right hand side of \eqref{Dyson series} satisfies the equation $\partial_t u =  \widetilde{\Delta} u $. Moreover by \eqref{weighted heat kernel real time} and \eqref{series estimate},  we have
$$ \int_M  \left|\sum_{k=0}^\infty \partial_t^k \beta_t^k(x, y) \right|  h(x) d\bar{\mu} < \infty,  $$
and
$$ \left|\sum_{k=0}^\infty \partial_t^k \beta_t^k(x, y) \right| \rightarrow 0 \hspace{0.5cm} \text{as} \hspace{0.5cm} t\rightarrow 0,$$
for $d(x, y)\ge \epsilon >0$. Thus the right hand side of \eqref{Dyson series} tends to the Delta function as $t\rightarrow 0$.
Then we have proved \eqref{Dyson series} and hence finished the proof.
\end{proof}
\begin{proof}[Proof of Corollary 1.3]
Under the assumption of Theorem 1.1, we have $h(x) \simeq 1$. Then by the fact $ \overline{V}(x, r) \simeq V(x, r) \simeq r^n $ and \eqref{weighted heat kernel real time}, we conclude
$$ C_1 t^{-\frac{n}{2}} \exp \left(- \frac{d^2(x, y)}{c_1t}  \right) \le \overline{p}_t(x, y) \le C_2 t^{-\frac{n}{2}} \exp \left(- \frac{d^2(x, y)}{c_2t}  \right).  $$
Thus according to \cite[Theorem 2.5]{GHL}, the result follows.
\end{proof}
\begin{proof}[Proof of Theorem 1.4]
By \eqref{Dyson series}, it is sufficient to consider $\nabla_x \partial_t^k \beta_t^k$. By the above argument, we have
$$ |\nabla_x \partial_t^k \beta_t^k (x, y) | = |\partial_t^k \nabla_x \beta_t^k (x, y)| \le \frac{k!}{(t\sin \frac{\pi}{4})^k } \max_{z\in \Gamma} | \nabla_x \beta_z^k (x, y)|.  $$
Note that by Proposition 2.2, we have
\begin{align*}
 |\nabla_x \beta_z^k(x, y) |  & \le \int_0^{|z|} \int_M |\nabla_x \bar{p}_{(|z|-s)e^{i \theta}} (x, w)|| \beta^{k-1}_{se^{i \theta}} (w, y)| \| \phi \|_\infty d\bar{\mu} ds \\
     & \le  \frac{C^{k-1} \| \phi \|_{\infty}^k}{(k-1)!} \int_0^{|z|} \frac{s^{k-1}}{1\wedge \sqrt{|z|-s} }  ds \int_M \bar{p}_{\alpha_1(|z|-s)}(x,w) \bar{p}_{\alpha_2 s}(w,y) d\bar{\mu} \\
     & \le C' \frac{C^{k-1} \| \phi \|_{\infty}^k}{(k-1)!} \frac{|z|^k}{1\wedge \sqrt{|z|}} \bar{p}_{\alpha s}(x,y),
\end{align*}
where $\alpha=\max\{ \alpha_1, \alpha_2 \}$. As a result, we have
$$ |\nabla_x \widetilde{p}_t(x, y) | \le \sum_{k\ge 0} |\partial_t^k \nabla_x \beta_t^k(x, y)| \le \sum_{k\ge 0} C_1 (\sqrt{2}C\| \phi \|_\infty)^k \frac{1}{1\wedge \sqrt{t}} \bar{p}_{\gamma t} (x, y). $$
Since $\widetilde{\nabla}= h^{-1}(x) \nabla$, we have proved the desired results.
\end{proof}
%\section*{Acknowledgment}
%The author is supported by the National Natural Science Foundation of China under Grant No.11901407 No.11971327 and the Fundamental Research Funds for the Central Universities No.2021SCU12105.

\bibliographystyle{amsplain}

\end{document}